\documentclass[11pt,letterpaper]{amsart}
 \usepackage[T1]{fontenc}
 \usepackage[cp1250]{inputenc}
\usepackage{xcolor}
\usepackage[normalem]{ulem}

 \newtheorem{theorem}{Theorem}
\newtheorem{cor}[theorem]{Corollary}
 \newtheorem{proposition}[theorem]{Proposition}
\theoremstyle{definition}
\newtheorem{definition}[theorem]{Definition}
\theoremstyle{remark}

\begin{document}

\title[$L^r$-differentiability of Two Lusin Classes]{On the $L^r$-differentiability of Two Lusin Classes and a Full Descriptive Characterization of the $HK_r$-integral}

\subjclass[2020]{Primary 26A39, 26A46}

\author[MUSIAL]{PAUL MUSIAL}
\address{Chicago State University\\
9501 South King Drive, Chicago, Illinois,  60628 USA}
\email{paul.musial@gmail.com}

\author[SKVORTSOV]{VALENTIN SKVORTSOV}
\address{Lomonosov Moscow State University, Mathematics Department
and Moscow Center of Fundamental and Applied Mathematics\\
Moscow, 119991 Russia}
\email{vaskvor2000@yahoo.com}

\author[SWOROWSKI]{PIOTR SWOROWSKI}
\address{Casimirus the Great University \\ Powsta\'nc\'ow Wielkopolskich 2, 85-090 Bydgoszcz, Poland}
\email{p.sworowski@gmail.com}

\author[TULONE]{FRANCESCO TULONE}
\address{University of Palermo\\
via Archirafi 34, 90123 Palermo, Italy}
\email{francesco.tulone@unipa.it}

\maketitle
\thispagestyle{empty}








\begin{abstract}
It is proved that any function of a Lusin-type class, the class of $ACG_r$-functions, is differentiable almost everywhere in the sense of a derivative defined in the space~$L^r$, $1\le r<\infty$. This leads to obtaining a full descriptive characterization of a Henstock-Kurzweil-type integral, the $HK_r$-integral, which serves to recover functions from their $L^r$-derivatives. The class $ACG_r$ is compared with the classical Lusin class $ACG$ and it is shown that a continuous $ACG$-function can fail to be $L^r$-differentiable almost everywhere.
\end{abstract}
\section{Introduction}
The well-known classical definition of the Lebesgue integral in terms of absolute continuity of the indefinite integral is an example of a so-called full descriptive characterization of an integral. Definitions of this type are known also for a number of non-absolute generalizations  of the Lebesgue integral. The most well-known among them is Lusin's definition of Denjoy integral (see~\cite{saks}).

In general, we say that a {\em descriptive} characterization of a constructive integration process is obtained if a class of functions is found so that it coincides with the class of the indefinite integrals of all functions integrable in the sense of this process, and if functions of the class are differentiable almost everywhere in a certain sense, the derivative being equal to the integrand almost everywhere. If the differentiability (in the appropriate sense, corresponding to differentiability property of the indefinite integral) of functions of the class is an additional assumption put on the class, then we have only a {\em partial} descriptive characterization of the considered integral. If on the other hand the corresponding differentiability property is a consequence of the description of the class and not  an additional requirement, such a characterization is said to be a {\em full} descriptive characterization. For example, each continuous function from Lusin's class $ACG$ is differentiable almost everywhere only in the sense of the approximate derivative, and this class gives a full descriptive characterization of the wide Denjoy integral (the $D$-integral, see~\cite{saks}). At the same time if we restrict this class  by considering only continuous and differentiable (in the usual sense) $ACG$-functions then we obtain a partial descriptive characterization of the so-called Khintchine integral (see~\cite{khin}). Descriptive characterizations of other non-absolute integrals are discussed in \cite{Ene, Pfeffer, PT, STHaar15, minmax17, Thomson}.

In this paper we obtain a full descriptive characterization of a Henstock-Kurzweil-type integral, the $ HK_r$-integral (defined in 2004 by Musial and Sagher~\cite{MusialSagher2004}), which serves to integrate a derivative defined in the space~$L^r$, $1\le r <\infty$. This $L^{r}$-derivative was introduced in \cite{Calderon Zygmund} by Calder{\'o}n and Zygmund to be used in some estimates for solutions of elliptic partial differential equations. The $HK_r$-integral was defined as an extension of a Perron-type integral, the $P_{r}$-integral, which was defined earlier by L.\,Gordon~\cite{gordon} and which also recovers a function from its $L^{r}$-derivative.  The $HK_{r}$-integral turned out to be strictly wider than the $P_{r}$-integral (see~\cite{MST1}). It was also shown in \cite{MusialSagher2004} that the indefinite $HK_{r}$-integral is $L^{r}$-differentiable almost everywhere and belongs to a Lusin-type class of $ACG_r$-functions. Some other properties of these integrals are investigated in \cite{MST3,mt15,mtdual19,Variational}.

Another class, also giving a descriptive characterization of the $HK_r$-integral, was obtained in \cite{MST2} in terms of absolute continuity of so-called $L^r$-variational measure generated by a function belonging to $L^r$. Both classes, the class of $ACG_r$-functions and the class of functions generating absolutely continuous $L^r$-variational measure, were known to coincide with the class of the indefinite $HK_{r}$-integrals, but only under the additional assumption that the functions are  $L^r$-differentiable almost everywhere (see~\cite{MST2}).  The problem of $L^r$-differentiability almost everywhere of functions of those classes was left open in~\cite{MST2}. So in fact only partial descriptive characterizations of the $HK_r$-integral were obtained in \cite{MusialSagher2004} and~\cite{MST2}.

The main aim of the present paper is to prove that any $ACG_r$-function is $L^{r}$-differentiable almost everywhere and  thereby obtain a full descriptive characterization of the $HK_r$-integral (see Theorem \ref{descr} in Section~\ref{s3}). In Section~\ref{s4}, we  construct an example of a function from a Lusin-type class, $[ACG]$, which, even under an additional assumption of continuity, fails to have the property of being $L^{r}$-differentiable almost everywhere for any $r$.  Thereby we show that the class $[ACG]$ is not contained in the class $ACG_r$ for any $r$. As a byproduct of this example we get a new proof, based on the $L^{r}$-derivative, of a result, previously obtained in~\cite{SS}, stating that the $HK_r$-integral does not cover the wide Denjoy integral.
\section{Preliminaries}
We recall here the main definitions and facts related to the notion of the $HK_r$-integral (see \cite{gordon,MusialSagher2004}).

We work in a fixed closed interval $[a,b]$. An interval $I$ is a non-degenerate closed subinterval of $[a,b]$. We assume in all these definitions that the functions $F$ belong to the class $L^r=L^r[a,b]$, $1\le r<\infty$. We start with $L^r$-derivates and the $L^r$-derivative.
\begin{definition}\label{derivate 1}
The {\em upper-right $L^r$-derivate} of $F$ at~$x$, denoted by $D_r^+F(x)$, is defined as the infimum of all numbers $\alpha$ such that
\begin{equation}\label{upper-right}
\left(\frac1h\int_0^h[F(x+t)-F(x)-\alpha t]_+^r\mkern1.5mu dt\right)^{\!1/r}\!=\,o(h)\quad\text{as}\ h\rightarrow 0^+.
\end{equation}
If no real number $\alpha$ satisfies~\eqref{upper-right}, we set $D_r^+F(x)=+\infty$. The {\em lower-right}, {\em upper-left}, and {\em lower-left $L^r$-derivates} of $F$ at $x$ are defined correspondingly as
\begin{gather*}
D_{+,r}F(x)=-D_r^+(-F)(x),\quad D_r^-F(x)=D_r^+(-F)(-x),\\
\text{and}\quad D_{-,r}F(x)=-D_r^+F(-x).
\end{gather*}
 If all four $L^r$-derivates of $F$ at $x$ are equal, the common value, denoted by $F'_r(x)$, is the {\em$L^r$-derivative of $F$ at~$x$}. If $F'_r(x) $ is finite we say that $F$ is {\em$L^{r}$-differentiable at $x$} and the value of $F'_r(x)$, say $\alpha$, is uniquely defined by
\begin{equation}\label{derivative}
\left(\frac1h\int_{-h}^h|F(x+t)-F(x)-\alpha t|^{r}\mkern1.5mu dt\right)^{\!1/r}\!=\,o(h)\quad\text{as}\ h\to 0^+.
\end{equation}
\end{definition}

We need the following property of $L^r$-derivates (see \cite[Corollary of Theorem 4]{gordon})
\begin{proposition}\label{prop}
If the two  upper (or the two lower) $L^r$-derivates of a function $F$ are less than $+\infty$ (resp. greater than $-\infty$) on a set~$E$, then almost everywhere on $E$ the $L^{r}$-derivative of $F$ exists and is finite.
\end{proposition}
For completeness we give the following definitions; see \cite{R Gordon} for a full treatment.
\begin{definition}
Let $F\colon[a,b]\to\mathbb R$ and let $c\in [a,b]$.  The function $F$ is said to be \textit{approximately continuous at} $c$ if there exists a measurable set $E$ containing $c$ such that $c$ is a point of density of $E$ and $F|_E$ is continuous at $c$.
\end{definition}
\begin{definition}
Let $F\colon[a,b]\to\mathbb R$ and let $c\in [a,b]$.  The function $F$ is said to be \textit{approximately differentiable at} $c$ if there exists a measurable set $E$ containing $c$ such that $c$ is a point of density of $E$ and
$$\lim_{x\rightarrow c, x\in E}\frac{F(x)-F(c)}{x-c}$$
 is finite. The value of this limit is the approximate derivative of $F$ at $c$, and we will denote it by $F'_{\textup{ap}}(c)$.  Note that the value of $F'_{\textup{ap}}(c)$ does not depend on the choice
of $E$.
\end{definition}
To define the $HK_r$-integral \cite{MusialSagher2004} we need some auxiliary notions. A  {\it tagged interval} is a pair $(I, x)$ where a point $x\in I$ is a {\it tag}. We say that tagged intervals  $(I',x')$ and $(I'',x'')$ are nonoverlapping if intervals $I'$ and $I''$ are nonoverlapping, i.e., they have no inner points in common. A {\em partition} is any finite collection $\pi$ of pairwise nonoverlapping tagged intervals. A {\it gauge} is a strictly positive function $\delta$ on $[a,b]$ (or on a subset of $[a,b]$). We say that a tagged interval $(I, x)$ is {\it $\delta$-fine} if $I\subset(x-\delta(x),x+\delta(x))$.

A partition is {\em$\delta$-fine} if all its elements are $\delta$-fine.
A partition $\pi$ is {\it tagged in a set} $E\subset [a,b]$ if $x \in E$ for each element  $(I, x)$ of $\pi$.

The Lebesgue measure on $[a,b]$ will be denoted by $\mu$. We recall the definition of the $L^{r}$-Henstock--Kurzweil integral given in \cite{MusialSagher2004}.
\begin{definition}\label{def1}
{\rm A function $f\colon[ a,b] \rightarrow \mathbb{R}$ is {\em $L^{r}$-Henstock--Kurzweil integrable} ({\em HK$_r$-integrable}) on $[ a,b] $ if there exists a function $F\in L^{r}$ such that for any $ \varepsilon >0$ there exists a gauge $\delta $ such that for any $\delta$-fine partition $\{([c_{i},d_{i}], x_i)\}_i$ we have
\begin{equation*}
\sum\limits_i\left( \dfrac{1}{d_{i}-c_{i}}\int_{c_{i}}^{d_{i}}\lvert F( y) -F( x_{i}) -f( x_{i}) (y-x_{i}) \rvert ^{r}\,dy\right)^{1/r}<\varepsilon.
\end{equation*}}
\end{definition}

\begin{definition}\label{def6} \cite{MusialSagher2004}
Let $E\subset[a,b]$ and $F\in L^r$. We say that $F\in AC_r(E)$ if for each $\varepsilon>0$ there exist $\eta>0$ and a gauge $\delta$ defined on $E$ such that for any  $\delta $-fine partition $\{([c_i,d_i], x_i)\}_i$ tagged in $E$ and such that
\begin{equation}\label{99}
\sum_i (d_i-c_i)<\eta
\end{equation}
we have
\begin{equation}\label{100}
\sum_i\left( \frac{1}{d_i-c_i}\int_{c_i}^{d_i}|F(y)-F(x_i)|^r\,dy\right)^{\!1/r}<\,\varepsilon.
\end{equation}
\label{def4}We say that $F$ is an $ACG_r$-\textit{function} if  $[a,b]=\bigcup_{n=1}^\infty E_n$ where $F\in AC_r(E_n)$ for all~$n$.
\end{definition}
\begin{definition}
Let $E\subset [a,b]$. We say that $F\in AC(E)$ if for each $\varepsilon>0$ there exist $\eta>0$ such that for any finite collection of nonoverlapping intervals  $\{[c_i,d_i]\}_i$ having endpoints in $E$ and such that
\begin{equation*}
\sum_i (d_i-c_i)<\eta
\end{equation*}
we have
\begin{equation*}
\sum_i|F(d_i)-F(c_i)|<\,\varepsilon.
\end{equation*}
We say that $F$ is an {\em ACG-function} if $[a,b]=\bigcup_{n=1}^\infty E_n$, where $F\in AC(E_n)$ for all~$n$. If, moreover, all $E_n$ can be chosen closed, we say $F$ is an {\em\textup[ACG\textup]-function}.
\end{definition}
\begin{definition}
A function $f\colon[a,b]\to\mathbb R$ is said to be {\em D-integrable} ({\em integrable in the wide Denjoy sense}) if there exists a {\em continuous} $ACG$-function $F\colon[a,b]\to\mathbb R$ such that $F'_{\textup{ap}}(x)=f(x)$ at almost all $x\in[a,b]$. We then define $\int_a^bf=F(b)-F(a)$.
\end{definition}
\begin{definition}
A function $f\colon[a,b]\to\mathbb R$ is said to be {\em Kubota integrable} if there exists an approximately continuous $[ACG]$-function $F\colon[a,b]\to\mathbb R$ such that \ $F'_{\textup{ap}}(x)=f(x)$ at almost all $x\in[a,b]$. We then define $\int_a^bf=F(b)-F(a)$.
\end{definition}
We will use the following property of $ACG_r$-functions.
\begin{proposition}\label{pr}
\cite[Corollary 1]{MusialSagher2004} If $F$ is an $ACG_r$-function, then  $[a, b]$ can be represented as $[a,b]=\bigcup_{n=1}^{\infty}E_n$ with  $F\in AC(E_n)$ for each~$n$.
\end{proposition}

Recall that $F\in L^r$ and so is measurable. Hence the above proposition implies that any $ACG_r$-function is a measurable $VBG$-function, and so by the Denjoy-Khintchine Theorem (see \cite[Chapter VII, Section 4, Theorem 4.3]{saks}) we obtain:
\begin{theorem}\label{appdiff}
Any $ACG_r$-function is approximately differentiable a.e.\ on $[a,b]$.
\end{theorem}
\section{$L_r$-differentiability of $ACG_r$-functions}\label{s3}
We now state and prove the main theorem of this paper.
\begin{theorem}\label{diff of ACGr}
Any $ACG_r$-function $F$ is $L_r$-differentiable a.e.\ on $[a,b]$.
\end{theorem}
\begin{proof}
According to Theorem \ref{appdiff}, $F'_{\textup{ap}}$ exists and is finite almost everywhere on $[a, b]$. We will show that
\begin{equation}\label{equal app}
F_r'(x)=F_{\textup{ap}}'(x)\ \ \text{a.e.}
\end{equation}
By Lusin's theorem on $C$-property it is enough to prove \eqref{equal app} on a closed set $E$ of positive measure such that $F$ restricted to $E$ is continuous. We can suppose that $F$ is approximately differentiable at each point of~$E$. Note that the approximate derivative and all $L_r$-derivates are measurable (see \cite{saks,gordon}). To simplify computation we consider first the case where $F'_{\textup{ap}}$ is non-negative on $E$. We start by proving that in this case
\begin{equation}\label{equal1}
D_{r}^{+}F(x)=F_{\textup{ap}}'(x) \ \ \text{a.e.\ on $E$}.
\end{equation}
If not, then having in mind the relation between $L^{r}$-derivates and approximate derivates (see \cite[Theorem~2]{gordon}) we get
\begin{equation}\label{inequal}
D_r^+F(x)>F_{\textup{ap}}'(x)\ge0
\end{equation}
on a set $E'\subset E$ of positive measure. We can suppose that $E'$ is closed. By Definition \ref{def4} we can find a set $E''\subset E'$ of positive outer measure where $F\in AC_r(E'')$. By Definition \ref{def6} for $\varepsilon=1$ there exist on $E''$ a gauge $\delta$ and a positive number $\eta$ such that inequality \eqref{100} holds for each $\delta $-fine partition $\{([c_i,d_i],x_i)\}_i$ tagged in~$E''$ and satisfying \eqref{99}. Now having fixed $\delta$ corresponding to the chosen $\varepsilon$ we can find a set $T \subset E''$ of positive outer measure on which $\delta(x)\geq \gamma$ for some constant $\gamma>0$. Let $\overline{T}$ be the closure of~$T$ and let $W$ be the set of two-sided limit points of $\overline{T}$.  We then have $\overline{T}=W\cup C$ where $C$ is countable. As the set $E'$ is closed we have $W\subset\overline{T}\subset E'\subset E$ and $\mu(W)>0$.

We show now that the inequality \eqref{100} with $\varepsilon=2$, i.e., the inequality
\begin{equation}\label{101}
\sum_i\left(\frac1{d_i-c_i}\int_{c_i}^{d_i}|F(y)-F(x_i)|^r\,dy\right)^{\!1/r}<\,2
\end{equation}
holds for any $\gamma$-fine partition $\{([c_i,d_i],x_i)\}_i$ tagged in $W$, and satisfying \eqref{99} with the same $\eta$ which was found for $\varepsilon=1$ and for $\delta$-fine partitions tagged in~$E''$. Indeed, consider such a $\gamma$-fine partition. By the definition of $W$ and by the continuity of $F$ with respect to~$E$,  we can choose, for each $i$, a point $x_i'\in T$ so close to $x_i$ {\it inside of} $[c_i,d_i]$ that  $([c_i,d_i],x_i')$ is also $\gamma$-fine and $\sum_i|F(x_i) - F(x'_i)|<1$. Then
\begin{align*}
&\sum_i\left( \frac{1}{d_i-c_i}\int_{c_i}^{d_i}|F(y)-F(x_i)|^r\mkern1.5mu dy\right)^{\!1/r}\\
&\qquad\le\sum_i\left( \frac{1}{d_i-c_i}\int_{c_i}^{d_i}|F(y)-F(x_i')|^r\mkern1.5mu dy\right)^{\!1/r}\\
&\qquad\qquad+\sum_i\left( \frac{1}{d_i-c_i}\int_{c_i}^{d_i}|F(x_i')-F(x_i)|^r\mkern1.5mu dy\right)^{\!1/r}\!<1+1=2.
\end{align*}
So \eqref{101} must be true for any $\gamma$-fine partition tagged in $W$ and satisfying~\eqref{99}. We now construct such a $\gamma$-fine partition for which \eqref{101} fails to be true. This will give us a contradiction proving~\eqref{equal1}.

Let $P\subset W$ be the set of density points of $W$. We have $\mu(P)=\mu(W)$. Let $x$ be any point of $P$. There exists a set $S_x$ having $x$ as a density point, such that
\begin{equation}\label{app}
\lim_{\substack{t\to 0\\x+t\in S_x}}\!\frac{F(x+t) - F(x)}{t}= F_{\textup{ap}}'(x)\ge0.
\end{equation}
Note that $x$ is also a density point for the set $W\cap S_x$. Moreover, we can find a closed set $A_x\subset W\cap S_x$
for which $x$ is still a density point. Let the open set $B_x$ be the complement of~$A_x$ in $[a,b]$. Choose the number $l$ so that $D^{+}_r(x)>l>F_{\textup{ap}}'(x)$. Then by Definition \ref{derivate 1} there exists a constant $c_x>0$ and a sequence $h_k\to 0^+$ such that for each $k$
\begin{equation}\label{not}
\left(\frac1{h_k}\int_{0}^{h_k}[F(x+t)-F(x)-lt]_+^r\mkern1.5mu dt\right)^{\!1/r}\!>\,c_x h_k.
\end{equation}
We now put some restrictions on the size of $h_k$. We cover $P$ by an open set $G\subset (a,b)$ such that $\mu(G)<2\mu(P)$, and we assume that all intervals $[x,x+h_k]$ are subsets of~$G$. As $A_x\cap[x,x+h]\subset\{x+t:F(x+t)-F(x)<lt\}$ for small enough~$h$, we can also assume that for each $h_k$
\begin{multline}\label{105}
\int_0^{h_k}[F(x+t)-F(x)-lt]_+^r\mkern1.5mu dt\\=\int_{B_x\cap[x, x+h_k]}[ F(u)-F(x)-l(u-x)]_+^r\mkern1.5mu du.
\end{multline}
Note that this equality in particular implies that in our further consideration we can assume that for all $k$,
\begin{equation}\label{>}
\mu(B_x\cap[x, x+h_k])>0.
\end{equation}
 Indeed, otherwise
$$\int_0^{h_k}[F(x+t)-F(x)-lt]_+^r\,dt= 0 $$
contradicting the claim that $D^{+}_rF(x)>l$.

Fix some natural $n>\max{\{2\mu(P)/\eta, 2/\mu(P)\}}$. As $x$ is a point of dispersion for the set $B_x$ we can find $h_x\le\gamma$ such that
\begin{equation}\label{106}
\frac{\mu(B_x\cap[x, x+h])}{h}< \min{\biggl\{\frac1n,\frac{c_x^r}{n^r}\biggr\}}
\end{equation}
holds for $0<h\le h_x$. So we can assume that $h_k<h_x\leq \gamma$ for all $k$. Having in mind \eqref{not},  \eqref{105}, \eqref{>}, and \eqref{106},  we obtain:
\begin{align}
&\left(\frac1{\mu(B_x\cap[x,x+h_k])}\int_{B_x\cap[x, x+h_k]}[F(u)-F( x) -l(u-x)]_+^r\mkern1.5mu du\right)^{\!1/r}\nonumber\\
&\quad=\left(\frac{h_k}{\mu(B_x\cap[x,x+h_k])}\right)^{\!1/r}\!\cdot\,\left(\frac1{h_k}\int_0^{h_k}[F(x+t)-F(x)-lt]_+^r\mkern1.5mu dt\right)^{\!1/r}\nonumber\\[-2ex]&\qquad\qquad\ge\biggl(\frac{n^r}{c^{r}_x}\biggr)^{\!1/r}\,\cdot c_xh_k=nh_k.\label{107}
\end{align}
As $B_x$ is an open set we have the representation
$$B_x\cap(x, x+h_k)= \bigcup_j\,(a_j, b_j).$$
Note that $a_j\in A_x\subset W$ and $b_j-a_j<h_x\le\gamma$ for each~$j$. So $([a_j,b_j],a_j)$ are $\gamma$-fine intervals tagged in~$W$.  Using the following obvious inequality for positive numbers
\[\left(\frac{\sum_ka_k}{\sum_kb_k}\right)^{\!1/r}\le\left(\sum_k\frac{a_k}{b_k}\right)^{\!1/r}\leq\sum_k\biggl(\frac{a_k}{b_k}\biggr)^{\!1/r}\]
(this is true for infinite sums provided the series are convergent), we get
\begin{align}\label{108}
\begin{split}
&\left(\frac1{\mu(B_x\cap[x,x+h_k])}\int_{B_x\cap[x,x+h_k]}[F(u)-F(x)-l(u-x)]_+^r\mkern1.5mu du\right)^{\!1/r}\\
&\qquad\le\left(\sum_j\frac1{b_j-a_j}\int_{a_j}^{b_j}[F(u)-F(x)-l(u-x)]_+^r\mkern1.5mu du\right)^{\!1/r}\\
&\qquad\le\left(\sum_j\frac1{b_j-a_j}\int_{a_j}^{b_j}|F(u) -F(a_j)|^r\mkern1.4mu du\right)^{\!1/r}\\
&\qquad\le\sum_j\left(\frac1{b_j-a_j}\int_{a_j}^{b_j}|F(u)-F(a_j)|^r\mkern1.5mu du\right)^{\!1/r}.
\end{split}
\end{align}
(We have taken into account that $F(a_j)- F(x)<l(a_j -x)<l(u-x)$ for $u \in (a_j, b_j)$). Combining \eqref{108} with \eqref{107} we get
\begin{equation}\label{110}
\sum_j\left(\frac 1{b_j-a_j}\int_{a_j}^{b_j}|F(u) -F(a_j)|^{r}\mkern1.5mu du\right)^{\!1/r}\geq nh_k.
\end{equation}
The family of intervals $\{[x, x+h_k]\}$ for all $x\in P$ and all $k$ forms a Vitali cover of~$P$. Applying the Vitali covering theorem we choose a countable sequence of nonoverlapping intervals $[x_j, x_j+h_{k_j}]$ which are subsets of $G$ and cover $P$ up to a set of measure zero. Hence $\mu(P)\le\,\sum_jh_{k_j}\le\mu(G)<2\mu(P)$. We can apply \eqref{110} to each of these intervals and to the corresponding set $B_{x_j}\cap[x_j, x_j+ h_{k_j}]$. Collecting all intervals constituting $B_{x_j}\cap[x_j, x_j+ h_{k_j}]$ for all $j$ we obtain a family $\{(I_i, z_i)\}_i$ of  nonoverlapping $\gamma$-fine intervals tagged in $P$, and so in $W,$  the total length of which is estimated, due to~\eqref{106}, by
$$\sum_j\mu(B_{x_j}\cap[x_j, x_j+ h_{k_j}]) < \sum_j\frac{ h_{k_j}}n \leq\frac{ 2\mu(P)}n<\eta.$$
So for any finite subfamily of
$\{(I_i,z_i)\}_i$ the inequality \eqref{101} should hold. At the same time by \eqref{110} and by the choice of $n$ we have
$$\sum_{i} \left(\frac 1{|I_{i}|}\int_{I_{i}}|F(u)-F(z_{i})|^r\,du\right)^{\!1/r}>n\mu(P)>2,$$
which must hold also
for a {\em certain} finite subfamily of $\{(I_i,z_i)\}_i$ (i.e., a
$\gamma$-fine partition tagged in $W$). This contradiction  proves that our assumption \eqref{inequal} is false and \eqref{equal1} is true for the case of non-negative $F'_{\textup{ap}}$.

In the same way using the definition of the upper-left $L^r$-derivate, we obtain a contradiction to the assumption that
\begin{equation}\label{inequa2}
D_{r}^{-}F(x)>F_{\textup{ap}}'(x)\ge0
\end{equation}
holds on a set of positive measure. In this case, instead of equality (\ref{105}), the equality
\begin{multline}\label{205}
\int_{0}^{h_k}[F(x-t)-F(x)+lt] _{-}^{r}\mkern1.5mu dt\\
=\int_{B_x\cup[x- h_k, x]}[F(u)-F(x)-l(u-x)] _{-}^{r}\mkern1.5mu du
\end{multline}
is used for respective $h_k$. This time the intervals $\{(a_i,b_i)\}_i$ constituting $B_x\cap[x-h_k,x]$ give rise to a  $\gamma$-fine partition $\{(I_i,z_i)\}_i$ tagged in~$W$. Eventually we obtain that:
\begin{equation}\label{equal2}
D_{r}^{-}F(x)=F_{\textup{ap}}'(x) \quad\text{a.e.\ on $E$}.
\end{equation}
(again in the case of non-negative $F'_{\textup{ap}}$). Having \eqref{equal1} and \eqref{equal2} we apply Proposition \ref{prop} and prove the statement of the theorem for the case of non-negative~$F'_{\textup{ap}}$.

If $F'_{\textup{ap}}$ is negative on a set of positive measure on which \eqref{equal app} is false we apply the argument above to the function $-F$.
\end{proof}
As we have already noted, the above theorem shows that the class of $ACG_r$-functions coincides with the class of indefinite integrals of all $HK_r$-integrable functions, and we obtain a full descriptive characterization of the $HK_r$-integral:
\begin{theorem}\label{descr}
A function $f$ is HK$_r$-integrable on $[a,b]$ if and only if there exists an $ACG_r$-function $F$ such that $F_r'=f$ a.e.; the function $F(x)-F(a)$ being the indefinite HK$_r$-integral of~$f$.
\end{theorem}
\section{The $ACG$ property and $L_r$-differentiability}\label{s4}
The following statement is proved in  \cite[Corollary~13]{SS}:
\begin{proposition}\label{prop111}
Each indefinite HK$_r$-integral is an \textup[ACG\textup]-function.
\end{proposition}
Using this proposition and equality (\ref{equal app}) for the indefinite $HK_r$-integral together with its approximate continuity (see \cite{MusialSagher2004}), we get:
\begin{theorem}\label{th2}
Every HK$_r$-integrable function on $[a,b]$  is Kubota integrable on $[a, b]$ and the values of these integrals coincide.
\end{theorem}
Once again using Proposition \ref{prop111} and Theorem \ref{descr} we obtain a slight improvement of Proposition \ref{pr}.
\begin{cor}\label{pr1}
 Each $ACG_r$-function is an \textup[ACG\textup]-function.
\end{cor}

We now show that an $ACG$-function, even a continuous one, may fail to be $L_r$-differentiable a.e. This, in particular, will imply that the inclusion stated in Corollary \ref{pr1} is proper.
\begin{theorem}\label{th}
There exists a continuous $ACG$-function which is not $L_r$-differentiable on a set of positive measure for any~$r$.
\end{theorem}
\begin{proof}
We let $r_{0,1}=[0,1]$. We remove from $r_{0,1}$ the open interval $u_{0,1}$ centered at $1/2$ having length $u_0=1/6$, leaving two closed intervals $r_{1,1}$ and $r_{1,2}$ each having length $r_1=5/12$.  We continue this process in such a way that at level $n$ we are left with $2^n$ closed intervals $\{r_{n,k}\}_{1\leq k\leq 2^n}$ each having length $r_n:=(n+4)/(2^{n+1}(n+2))$.  In order to accomplish this, from each interval $r_{n,k}$ the concentric open interval $u_{n,k}$ of length $u_n:=1/(2^{n}(n+2)(n+3))$ is removed. We have
$$\left|\bigcup_{k=1}^{2^n} r_{n,k}\right|=\frac12+\frac{1}{n+2}$$
Let
$$P=\bigcap_{n=1}^{\infty}\bigcup_{k=1}^n r_{n,k}$$
$P$ is a symmetric, perfect, Cantor-like subset of $[0,1]$ having measure $1/2$. Let $v_{n,k}$ be the open interval concentric to $u_{n,k}$ having length $v_n:=u_n/2$. Define a function $F$ taking the value $0$ on $P$ and $(-1)^n/n$ on $v_{n,k}$.  Extend $F$ linearly on the subintervals of $u_{n,k}$ adjacent to~$v_{n,k}$.

$F$ is obviously continuous. It is $AC(P)$ and $AC(u_{n,k})$ for each $n$ and $k$, and so it is $ACG$ on the whole interval $[0,1]$.

We will now show that $F$ fails to be  $L_r$-differentiable at any point of $P$ which is not right-hand isolated in $P$.  Let $x$ be such a point. There are arbitrarily large $n\in\mathbb N$ such that for some $h>0$ and $k$, $x+h$ is the midpoint of $v_{n,k}$. Let $l>0$ be such that $x+l$ is the left endpoint of such $v_{n,k}$ and let $\alpha\in\mathbb{R}$.   One can estimate
\begin{multline*}
\frac1h\left(\frac1h\int_0^h|F(x+t)-F(x)-\alpha t|^r\,dt\right)^{\!1/r}\\
>\frac1h\left(\frac1h\int_l^h|F(x+t)-\alpha t|^r\,dt\right)^{\!1/r}\ge\frac1h\left(\frac1h\int_l^h|F(x+t)|^r\,dt\right)^{\!1/r}
\end{multline*}
when $\alpha\ge0$ and $n$ is odd or $\alpha\le0$ and $n$ is even. Continue the estimate from below:
\begin{align*}
&\ge\frac2{r_n}\left(\frac2{r_n}\cdot\frac{v_n}{2n^r}\right)^{\!1/r}=\frac{2^{1-1/r}}{nr_n}\left(\frac{u_n}{r_n}\right)^{\!1/r}\\
&=2^{n+2}\cdot\frac{n+2}{n(n+4)^{1+1/r}(n+3)^{1/r}},
\end{align*}
which is unbounded in $n$.  This proves that the function $F$ fails to be $L_r$-differentiable on a set of measure $1/2$.
\end{proof}
As we have already mentioned the constructed example gives another proof of the following result obtained in \cite[Theorem 14]{SS}:
\begin{cor}\label{cor}
There exists a function which is D-integrable but which is HK$_r$-integrable  for no~$r$.
\end{cor}
Indeed, for $F$ from the proof of Theorem \ref{th} suppose $f=F'_\textup{ap}$ is $HK_r$-integrable. Then the indefinite $HK_r$-integral $G=\int f$ is an $[ACG]$-function (Theorem \ref{th2}) with $G'_r(x)=f(x)$ a.e.\ in $[0,1]$ (Theorem \ref{descr}). Since $G'_\textup{ap}(x)=f(x)$ a.e.\ in $[0,1]$, due to monotonicity property of the class of approximately continuous $[ACG]$-functions, $F-G$ is constant, so that $F$, like $G$, is a.e.\ $L^r$-differentiable, giving a contradiction.

\end{document}